\documentclass{elsarticle}

\usepackage{natbib}
\usepackage{mathtools}
\usepackage{amsthm}
\usepackage{amsmath}
\usepackage{hyperref}

\newtheorem{theorem}{\bf Theorem}[section]

\newtheorem{definition}{\bf Definition}[section]

\newtheorem{lemma}[theorem]{\bf Lemma}
\newtheorem{proposition}[theorem]{\bf Proposition}

\begin{document}
\begin{frontmatter}

\title{The Territorial Raider Game and Graph Derangements}

\author{%
Nina Galanter}
\address{%
Grinnell College, Grinnell, IA 50112, USA}

\author{%
Dennis Silva, Jr.}
\address{Department of Mathematical Sciences, Worcester Polytechnic Institute, Worcester, MA 01609, USA}

\author{%
Jonathan T. Rowell}
\address{Department of Mathematics and Statistics
The University of North Carolina at Greensboro, Greensboro, NC 27412, USA}

\author{%
Jan Rycht\'a\v{r}}
\address{Department of Mathematics and Statistics
The University of North Carolina at Greensboro, Greensboro, NC 27412, USA}

\begin{abstract}
A derangement of a graph $G=(V,E)$ is an injective function  $f:V\to V$ such that for all $v\in V$, $f(v)\neq v$ and $(v,f(v))\in E$. Not all graphs admit a derangement and previous results have characterized graphs with derangements  using neighborhood conditions for subsets of $V$.  We establish an alternative criterion for the existence of derangements on a graph. We analyze strict Nash equilibria  of the biologically motivated Territorial Raider Game, a multi-player competition for resources in a spatially structured population based on animal raiding and defending behavior. We find that a graph $G$ admits a derangement if and only if there is a strict Nash equilibrium of the Territorial Raider game on $G$.
\end{abstract}


\end{frontmatter}

\section{Introduction}
A set derangement is a permutation of a set's elements with no fixed points \cite{Hassani}. Similarly, a graph derangement is a permutation of the vertices of a graph which has no fixed points, with additional limitations imposed by the structure of the graph. More formally, a graph derangement is an injective function mapping all vertices of a graph to adjacent vertices \cite{Clark}. While derangements exist for all sets containing more than one element, the existence of a derangement of a graph depends on its structure \cite{Hassani}.

Tutte (1953) introduced the idea of the Q-factor of an unoriented graph \cite{Tutte}. A Q-factor is a spanning subgraph which consists of 1-regular  components (vertex pairs), and 2-regular components (cycles). A graph will have a Q-factor if and only if that graph has a derangement \cite{Clark}.

A finite graph $G$ with vertex set $V$ admits a derangement, or equivalently, has a Q-factor, if and only if, for any finite subset $W\subseteq V$, $|N(W)|\geq|W|$, where $N(W)$ is the set of all vertices adjacent to a vertex in $W$ \cite{Clark,Tutte}.

In this paper, we provide a new criterion for the existence of graph derangements.
We adapt the Territorial Raider game (see for example \cite{broom2015study, broom2012general, BroomRychtarPreprint, bruni2013analysing})
and establish a one-to-one correspondence between a derangement of a graph and a strict Nash equilibrium of the Territorial Raider game. This will prove that a simple, finite, undirected, connected graph $G$ admits a derangement if and only if a Territorial Raider game played on $G$ has a strict Nash equilibrium. 
This game theoretical approach allows graphs to be analyzed through the implementation of multi-agent machine-learning algorithms
such as Exp3 \cite{auer1995gambling} which can potentially determine Nash equilibria, and thus the existence of derangements \cite{Hu}.

\section{Notation and preliminaries}

In this paper, any graph referred to is assumed to be simple, finite, undirected, and connected.
\begin{definition}
Let $G=(V,E)$ be a graph. A \emph{derangement} of $G$ is an injective function  $f:V\to V$ such that for all $v\in V$, $f(v)\neq v$ and $(v,f(v))\in E$.
\end{definition}

The \emph{Territorial Raider game}  is played on a graph $G=(V,E)$. Every vertex $v\in V$ is occupied by a player $I_v$ and the vertex $v$ is called the \emph{home vertex} of $I_v$. All vertices contain one unit of resources. All players must simultaneously choose whether to raid a neighboring vertex or stay home to defend against potential raiders. The object of the game is to maximize the resources obtained.

We are interested in determining the \emph{strict Nash equilibria}, which are sets of the strategies of all players such that any player will reduce their payoff by unilaterally changing their strategy \cite{Fudenberg}. We note that strict Nash equilibria must consist of pure strategies, see for example \cite{Fudenberg, Harsanyi}.
Consequently, we will only consider pure strategies.

Formally, a \emph{strategy} for player $I_v$
is a choice of a vertex $w\in V$ such that $w=v$ or $(v,w)\in E$. An \emph{admissible function} of $G$ is a function $f:V\to V$ such that for all $v\in V$, $f(v)=v$ or $(v,f(v))\in E$. We can see that there is a one-to-one correspondence between strategy sets for the players and admissible functions of $G$. We will use $f^{-1}$ to denote the inverse of $f$ when $f$ is bijective, and the preimage of $f$ otherwise.

When all individuals move according to their strategy, they receive payoffs based on their position as well as the positions of their opponents. By staying home, a player guarantees their claim to a portion $h\in[0,1]$ of their resources. The remaining resources are then split equally between the occupants of a vertex. If a player raids, they lose all of their home resources to raiders unless no other player raids their home vertex, in which case they keep all of their home resources.

Specifically, if $f(v)=v$ (a player $I_v$ chooses to defend), that player will receive
\begin{equation} \label{eq:payoff home}
P_v(f) = h + \frac{(1-h)}{|f^{-1}(v)|}
\end{equation}
where  $|f^{-1}(v)|$ denotes the cardinality of the preimage of $f$, and thus the total number of players at vertex $v$. If $f(v)=v'\neq v$ (a player $I_v$ chooses to raid node $v'$), then the payoff is

\begin{equation} \label{eq:payoff raid}
P_v(f) =
   \begin{cases}
       1+ \frac{1-h}{|f^{-1}(v')|}  & \text{if $f^{-1}(v)=\emptyset, f(v')=v'$ (no player raids $v$ and $I_{v'}$ defends)} \\
       1+\frac{1}{|f^{-1}(v')|}       & \text{if $f^{-1}(v)=\emptyset, f(v')\neq v'$ (no player raids $v$ and $I_{v'}$ raids)} \\
      \frac{1-h}{|f^{-1}(v')|}   &\text{if $f^{-1}(v)\neq \emptyset, f(v')=v'$ (some player raids $v$ and $I_{v'}$ defends)} \\
      \frac{1}{|f^{-1}(v')|}  &\text{if $f^{-1}(v)\neq \emptyset, f(v')\neq v'$ (some player raids $v$ and $I_{v'}$ raids)}
    \end{cases}
\end{equation}

We note that in order for a strict Nash equilibrium to exist, we must have $h<1$. Indeed, for a contradiction, assume $h=1$ and that the Nash equilibrium is generated by $f:V\to V$. If $f(v) = v$ for all $v\in V$, then any individual can raid a neighbor and its payoff stays the same. If there is $v\in V$ such that $f(v)\neq v$, then individual $I_v$ can stay home, receive the payoff of $1$, and thus not reduce its payoff.

The main result of this paper is the following theorem.
\begin{theorem} \label{thm:main}
A  simple, finite, undirected, and connected graph $G$ admits a derangement if and only if a Territorial Raider game played on $G$ , with $h\in[0,1)$, has a strict Nash equilibrium strategy set.
\end{theorem}

\section{Proof of Theorem \ref{thm:main}}

We will first show that any derangement generates a strict Nash equilibrium (Proposition \ref{prop:derangement generates strict NE}).
Then, we will show that any strict Nash equilibrium must be generated by a derangement (Theorem \ref{thm:all equilibria come from derangements}).
Theorem \ref{thm:main} follows directly from these two results.

\begin{proposition} \label{prop:derangement generates strict NE}
If $h\in [0,1)$, then any a derangement $f$ of $G$ generates a strict Nash equilibrium of the territorial raider game on $G$.
\end{proposition}

\begin{proof}
Under the definition of a derangement, a player $I_v$ moved to $f(v)\neq v$, $I_v$ is the only occupant of $f(v)$, and the vertex $v$ is itself raided by player $I_{f^{-1}(v)}$. The current payoff to player $I_v$ is therefore 1.
So, if player $I_v$ changes its strategy to defend, it will receive $h+\frac{1-h}{2}<1$ instead of the current payoff of $1$. If the player $I_v$ decides to raid a different vertex $w$ (which is already raided by $I_{f^{-1}(w)}$), it will receive a payoff of $\frac{1}{2}<1$. Consequently, $f$ generates a strict Nash equilibrium.
\end{proof}

In order to simplify the proof of Theorem \ref{thm:all equilibria come from derangements}, we will first prove the following Lemma \ref{lem:aux}.

\begin{lemma} \label{lem:aux}
If $f:V\to V$ generates a strict Nash equilibrium, then we cannot have all three of the following conditions satisfied:
\begin{enumerate}
\item \label{itm: 1} $f(v) = w\neq v$ ($I_v$ raids $w$),
\item \label{itm: 2} $|f^{-1}(w)| >1$ (there is another player on $w$; either $I_w$ or another raider), and
\item \label{itm: 3} $|f^{-1}(v)|\leq 1$, (the vertex $v$ itself is raided by at most one individual).
\end{enumerate}
\end{lemma}

\begin{proof}
Under assumptions (\ref{itm: 1}) and (\ref{itm: 2}), the payoff to individual $I_v$ is no more than $\frac{1}{2}$. If only one player raids $v$ and $I_v$ changes its strategy to stay home, the payoff for $I_v$ will be $h + \frac{1-h}{2}\geq \frac{1}{2}$. Thus $f$ does not generate a strict Nash equilibrium.

If $v$ is not raided by any player, then the player $I_w$ can receive the payoff of $1$ (the maximal payoff in this game) by raiding $v$ (note that $I_v$ raided $w$ and so $v$ and $w$ are connected). Again, $f$ does not generate a strict Nash equilibria.
\end{proof}

Now, we can prove the second implication in Theorem \ref{thm:main}.

\begin{theorem}\label{thm:all equilibria come from derangements}
For any $h\in [0,1)$, every strict Nash equilibrium of the Territorial Raider game  on $G$ is generated by a derangement of $G$.
\end{theorem}

\begin{proof}
Let $f:V\to V$ be a function that generates a strict Nash equilibrium.
First, let us show that $f$ must be injective. For a contradiction, let there be $v_0\in V$ such that $|f^{-1}(v_0)|>1$.
Set $H_0 = \{v_0\}$  and let $H_1= f^{-1}(v_0)\setminus \{v_0\}$ be the set of home vertices of the raiders of $v_0$. Note that $|H_1|\geq 1$ and $H_0\cap H_1 = \emptyset$.
By Lemma \ref{lem:aux}, as $f$ generates a strict Nash equilibrium and satisfies conditions (\ref{itm: 1}) and (\ref{itm: 2}), condition (\ref{itm: 3}) cannot be satisfied. Thus, each vertex from $H_1$ must be raided by at least two individuals.
Set $H_2 = f^{-1}(H_1)\setminus\{v_0\}$.
Since individuals from $H_1$ raid $v_0$ (and thus do not stay at their home vertices), $H_1\cap H_2 = \emptyset$. Therefore, the sets $\{H_i\}_{i=0}^2$ are pairwise disjoint.
Again, by Lemma \ref{lem:aux}, $|H_2|\geq 2|H_1|-1\geq 1.$
By induction, once $H_k$ is given, we can set
$H_{k+1} = f^{-1}(H_k)\setminus \{v_0\}$ to get a  sequence of pairwise disjoint sets $\{H_i\}_{i=0}^\infty$ such that each vertex in $H_{k}$ is raided by individuals  from $H_{k+1}$, and by Lemma \ref{lem:aux}, $|H_{k+1}|\geq 2|H_k|-1 \geq 1.$
Such a sequence of sets is impossible in a finite graph and thus we have a contradiction with our assumption that $f$ is not injective.

Now, assume that $f$ is injective but that there is $v\in V$ such that $f(v)=v$.
Since $f$ is injective, no player raids $v$ and thus $I_v$ would improve its payoff by raiding any of its neighboring vertices. This is a contradiction with the assumption that $f$ generates a strict Nash equilibrium.
\end{proof}

Note that since any derangement of $G=(V,E)$ is
a bijective function $f$ with $f(v)\neq v$ for all $v\in V$, there are at most  $\left[\frac{|V|!}{e}\right]$ such derangements (and consequently at most $\left[\frac{|V|!}{e}\right]$ strict equilibria of the Territorial Raider game on $G$), where $[\cdot]$ indicates the nearest integer function \cite{Hassani}.
The upper bound is attained for a complete graph.

\section*{Acknowledgements}
N.G. and D.S. were undergraduate student participants of the  summer 2015 UNCG Research Experience for Undergraduate program  supported by the NSF Grant \#1359187.
The authors acknowledge the guidance of Dr. Olav Rueppell, Quinn Morris, and Catherine Payne. The authors also acknowledge  Dr. Clifford Smyth and Dr. Dewey Taylor for useful discussions on background in graph theory.

\section*{References}
\bibliographystyle{model2-names}
\bibliography{ref}

\begin{thebibliography}{11}
\expandafter\ifx\csname natexlab\endcsname\relax\def\natexlab#1{#1}\fi
\providecommand{\url}[1]{\texttt{#1}}
\providecommand{\href}[2]{#2}
\providecommand{\path}[1]{#1}
\providecommand{\DOIprefix}{doi:}
\providecommand{\ArXivprefix}{arXiv:}
\providecommand{\URLprefix}{URL: }
\providecommand{\Pubmedprefix}{pmid:}
\providecommand{\doi}[1]{\href{http://dx.doi.org/#1}{\path{#1}}}
\providecommand{\Pubmed}[1]{\href{pmid:#1}{\path{#1}}}
\providecommand{\bibinfo}[2]{#2}
\ifx\xfnm\relax \def\xfnm[#1]{\unskip,\space#1}\fi
\bibitem[{Auer et~al.(1995)Auer, Cesa-Bianchi, Freund and
  Schapire}]{auer1995gambling}
\bibinfo{author}{Auer, P.}, \bibinfo{author}{Cesa-Bianchi, N.},
  \bibinfo{author}{Freund, Y.}, \bibinfo{author}{Schapire, R.E.},
  \bibinfo{year}{1995}.
\newblock \bibinfo{title}{Gambling in a rigged casino: The adversarial
  multi-armed bandit problem}, in: \bibinfo{booktitle}{Foundations of Computer
  Science, 1995. Proceedings., 36th Annual Symposium on},
  \bibinfo{organization}{IEEE}. pp. \bibinfo{pages}{322--331}.
\bibitem[{Broom et~al.(2015)Broom, Lafaye, Pattni and
  Rycht{\'a}{\v{r}}}]{broom2015study}
\bibinfo{author}{Broom, M.}, \bibinfo{author}{Lafaye, C.},
  \bibinfo{author}{Pattni, K.}, \bibinfo{author}{Rycht{\'a}{\v{r}}, J.},
  \bibinfo{year}{2015}.
\newblock \bibinfo{title}{A study of the dynamics of multi-player games on
  small networks using territorial interactions}.
\newblock \bibinfo{journal}{Journal of Mathematical Biology} ,
  \bibinfo{pages}{1--24}\DOIprefix\doi{10.1007/s00285-015-0868-1}.
\bibitem[{Broom and Rycht{\'a}{\v{r}}(2012)}]{broom2012general}
\bibinfo{author}{Broom, M.}, \bibinfo{author}{Rycht{\'a}{\v{r}}, J.},
  \bibinfo{year}{2012}.
\newblock \bibinfo{title}{A general framework for analysing multiplayer games
  in networks using territorial interactions as a case study}.
\newblock \bibinfo{journal}{Journal of Theoretical Biology}
  \bibinfo{volume}{302}, \bibinfo{pages}{70--80}.
\bibitem[{Broom and Rycht{\'a}{\v{r}}(2015)}]{BroomRychtarPreprint}
\bibinfo{author}{Broom, M.}, \bibinfo{author}{Rycht{\'a}{\v{r}}, J.},
  \bibinfo{year}{2015}.
\newblock \bibinfo{title}{Ideal cost-free distributions in structured
  populations for general payoff functions}.
\newblock \bibinfo{journal}{preprint} .
\bibitem[{Bruni et~al.(2013)Bruni, Broom and
  Rycht{\'a}{\v{r}}}]{bruni2013analysing}
\bibinfo{author}{Bruni, M.}, \bibinfo{author}{Broom, M.},
  \bibinfo{author}{Rycht{\'a}{\v{r}}, J.}, \bibinfo{year}{2013}.
\newblock \bibinfo{title}{Analysing territorial models on graphs}.
\newblock \bibinfo{journal}{Involve, a Journal of Mathematics}
  \bibinfo{volume}{7}, \bibinfo{pages}{129--149}.
\bibitem[{Clark(2013)}]{Clark}
\bibinfo{author}{Clark, P.L.}, \bibinfo{year}{2013}.
\newblock \bibinfo{title}{Graph derangements}.
\newblock \bibinfo{journal}{Open Journal of Discrete Mathematics}
  \bibinfo{volume}{3}, \bibinfo{pages}{183--191}.
\bibitem[{Fudenberg and Tirole(1991)}]{Fudenberg}
\bibinfo{author}{Fudenberg, D.}, \bibinfo{author}{Tirole, J.},
  \bibinfo{year}{1991}.
\newblock \bibinfo{title}{Game theory}. volume \bibinfo{volume}{393}.
\newblock \bibinfo{publisher}{Cambridge, Massachusetts}.
\bibitem[{Harsanyi(1973)}]{Harsanyi}
\bibinfo{author}{Harsanyi, J.C.}, \bibinfo{year}{1973}.
\newblock \bibinfo{title}{Games with randomly disturbed payoffs: A new
  rationale for mixed-strategy equilibrium points}.
\newblock \bibinfo{journal}{International Journal of Game Theory}
  \bibinfo{volume}{2}, \bibinfo{pages}{1--23}.
\bibitem[{Hassani(2003)}]{Hassani}
\bibinfo{author}{Hassani, M.}, \bibinfo{year}{2003}.
\newblock \bibinfo{title}{Derangements and applications}.
\newblock \bibinfo{journal}{Journal of Integer Sequences} \bibinfo{volume}{6}.
\bibitem[{Hu and Wellman(1998)}]{Hu}
\bibinfo{author}{Hu, J.}, \bibinfo{author}{Wellman, M.P.},
  \bibinfo{year}{1998}.
\newblock \bibinfo{title}{Multiagent reinforcement learning: theoretical
  framework and an algorithm}, in: \bibinfo{booktitle}{ICML}, pp.
  \bibinfo{pages}{242--250}.
\bibitem[{Tutte(1953)}]{Tutte}
\bibinfo{author}{Tutte, W.}, \bibinfo{year}{1953}.
\newblock \bibinfo{title}{The 1-factors of oriented graphs}.
\newblock \bibinfo{journal}{Proceedings of the American Mathematical Society}
  \bibinfo{volume}{4}, \bibinfo{pages}{922--931}.

\end{thebibliography}

\end{document}